\documentclass[12pt,reqno]{amsart}

\usepackage{amsfonts,amsmath,amsthm}
\usepackage{amssymb,epsfig}
\usepackage{cases}
\usepackage[usenames]{color}
\usepackage{enumerate} 

\usepackage[text={425pt,650pt},centering]{geometry}
\usepackage{graphicx}
\usepackage{epsfig}
\usepackage{caption}
\usepackage{tikz}
\usepackage{color} 
\definecolor{vert}{rgb}{0,0.6,0}

\theoremstyle{plain}
\newtheorem{thm}{Theorem}[section]

\newtheorem{defn}{Definition}

\newtheorem{ex}{Example}
\newtheorem{lem}[thm]{Lemma}
\newtheorem{cor}[thm]{Corollary}
\newtheorem{prop}[thm]{Proposition}
\theoremstyle{remark}
\newtheorem{rem}{\bf{Remark}}
\numberwithin{equation}{section}

\newcommand{\R}{\mathbb{R}}

\newcommand{\T}{\mathbb{T}}
\newcommand{\Z}{\mathbb{Z}}

\newcommand{\cL}{\mathcal{L}}
\newcommand{\cM}{\mathcal{M}}

\newcommand{\cS}{\mathcal{S}}

\newcommand{\cR}{\mathcal{R}}

\newcommand{\Lip}{{\rm Lip\,}}

\newcommand{\al}{\alpha}

\newcommand{\del}{\delta}
\newcommand{\ep}{\varepsilon}

\newcommand{\lam}{\lambda}
\newcommand{\sig}{\sigma}

\newcommand{\Del}{\Delta}

\newcommand{\Lam}{\Lambda}

\newcommand{\ol}{\overline}

\newcommand{\Div}{{\rm div}\,}

\begin{document}

\title[Generalized ergodic problems]{Generalized ergodic problems: existence and uniqueness structures of solutions}

\author{Wenjia Jing}
\address[W. Jing]{
Yau Mathematical Sciences Center,
Tsinghua University, No.1 Tsinghua Yuan,
Beijing 100084, China }
\email{wjjing@tsinghua.edu.cn}

\author{Hiroyoshi Mitake}
\address[H. Mitake]{
Graduate School of Mathematical Sciences, 
University of Tokyo 
3-8-1 Komaba, Meguro-ku, Tokyo, 153-8914, Japan}
\email{mitake@ms.u-tokyo.ac.jp}

\author{Hung V. Tran}
\address[Hung V. Tran]
{
Department of Mathematics, 
University of Wisconsin Madison, Van Vleck hall, 480 Lincoln drive, Madison, WI 53706, USA}
\email{hung@math.wisc.edu}


\maketitle

\begin{abstract}
We study a generalized ergodic problem (E), which is a Hamilton-Jacobi equation of contact type, in the flat $n$-dimensional torus.
We first obtain existence of solutions to this problem under quite general assumptions.
Various examples are presented and analyzed to show that (E) does not have unique solutions in general.
We then study uniqueness structures of solutions to (E) in the convex setting by using the nonlinear adjoint method.
\end{abstract}


\section{Introduction}

In this paper, we focus on the following equation
\[
{\rm (E)} \qquad
H(x,u,Du) = c \qquad \text{ in } \T^n.
\]
Here, $\T^n =\R^n/\Z^n$ is the flat $n$-dimensional torus,
and the Hamiltonian $H=H(x,r,p):\T^n \times \R \times \R^n \to \R$ is a given continuous function.
We seek for a pair of unknowns $(u,c) \in C(\T^n) \times \R$ that solves (E) in the viscosity sense.
We use $Du$ to denote the spatial gradient of $u$.
We are always concerned  with viscosity solutions, and the adjective ``viscosity" is often omitted in the paper.

Our main goals in this paper are twofold. First of all, we obtain existence results of solutions to (E) under quite general assumptions.
Second, it is well-known in the theory of viscosity solutions that if $r \mapsto H(x,r,p)$ is not strictly monotone, then (E) might not have unique solutions (see Examples \ref{ex-4}--\ref{ex-7} in Section \ref{sec:ex}).
It is therefore of our main interests to understand why this phenomenon appears, and to describe uniqueness structures of solutions to (E).

We call (E) a generalized ergodic problem.
In various other contexts, (E) is also called a Hamilton-Jacobi equation of contact type. 

\subsection{Assumptions}
We list here the main assumptions on Hamiltonian $H$ that are used in the paper.
\begin{itemize}
\item[(H1)] $H$ is uniformly Lipschitz in $r$, that is, there exists a constant $C_1>0$ such that
\[
|H(x,r,p) - H(x,s,p)| \leq C_1 | r-s| \quad \text{for all } (x,p) \in \T^n \times  \R^n, \, r,s \in \R.
\]

\item[(H2a)] $H$ is coercive in $p$, that is,
\[
\lim_{|p| \to \infty} H(x,0,p)=+\infty \quad \text{ uniformly for } x \in \T^n.
\]

\item[(H2b)] $H$ is superlinear in $p$, that is,
\[
\lim_{|p| \to \infty} \frac{H(x,0,p)}{|p|}=+\infty \quad \text{ uniformly for } x \in \T^n.
\]
\end{itemize}
It is clear that (H2b) is stronger than (H2a). We will assume either (H2a) or (H2b) in each of our results on existence of solutions to (E).
To address the uniqueness structure, we need to assume the following assumptions.

\begin{itemize}

\item[(H3)] $H \in C^2(\T^n \times \R \times \R^n)$, and
\[
\lim_{|p| \to \infty} \left(\frac{1}{2}H(x,r,p)^2 + D_xH(x,r,p)\cdot p\right)=+\infty \quad \text{ uniformly for } (x,r) \in \T^n \times \R.
\]
\item[(H4)] The map $r \mapsto H(x,r,p)$ is nondecreasing for all $(x,p) \in \T^n \times  \R^n$.

\item[(H5)] The map $(r,p) \mapsto H(x,r,p)$ is convex for all   $x \in \T^n$.
\end{itemize}
It is worth noting that (H3) and (H4) are quite standard assumptions. 
We only require that $H$ is nondecreasing in $r$ in (H4), so it may fail to be strictly increasing.
Condition (H5) however is rather strong since convexity is imposed both in $r$ and $p$.
In any case, nowhere in this paper do we require $H$ to be uniformly convex in $p$.

\subsection{Main results}
We first state two existence results for solutions to (E).
The first one is quite a standard result in light of the classical Perron method.
\begin{thm}\label{thm:exist1}
Assume {\rm (H1), (H2a)}.
Assume further that there exist $c\in \R$, and $\psi, \varphi \in \Lip(\T^n)$ such that $\psi \leq \varphi$,
$\psi$ and $\varphi$ are a viscosity subsolution and a viscosity supersolution to {\rm (E)}, respectively. 
Then, {\rm (E)} has a viscosity solution $u\in \Lip(\T^n)$ with $c\in \R$ given by the assumptions.
\end{thm}
This result is not new in the literature, and is just a variant of the classical results in \cite{Is}.
What is different here is that under assumptions (H1) and (H2a), we obtain directly a Lipschitz viscosity solution $u$ with known Lipschitz constant,
which is not written down explicitly in \cite{Is}.
It is therefore of our interests to record it here.

Next is our second existence result for solutions to (E) without prior information about the constant $c$.

\begin{thm}\label{thm:exist2}
Assume {\rm (H1), (H2b)}.
Then, {\rm (E)} has a solution $(v,c) \in \Lip(\T^n) \times \R$.
\end{thm}
As we do not assume the existence of a subsolution $\psi$ and a supersolution $\phi$ with $\psi\le\phi$ for some given $c\in \R$ as in Theorem \ref{thm:exist1}, the existence of solutions to (E) cannot be obtained by the standard Perron method. 
Existence result for (E) was obtained in \cite[Theorem 1.5]{WWY-0} under an additional assumption that $H$ is uniformly convex.
See also \cite{SWY}.
Unlike \cite{SWY, WWY-0}, we do not need any convexity of $H$ here, and we believe that Theorem \ref{thm:exist2} is new in the literature.

We emphasize here that, although the existence of $(v,c) \in C(\T^n) \times \R$, solution to (E), is guaranteed by  Theorems \ref{thm:exist1}--\ref{thm:exist2},
we do not have uniqueness of constant $c$ in general.
See Examples \ref{ex:2}, \ref{ex-4}, and \ref{ex-5} below.
Furthermore, for a fixed $c\in \R$ such that (E) has a solution, (E) might have multiple solutions as described explicitly in Section \ref{sec:ex} (Examples \ref{ex-4}--\ref{ex-7}). 
It is therefore extremely important to proceed further to understand this phenomenon and investigate how such nonuniqueness appears.
In particular, we aim to find a uniqueness set of (E), that is a set (hopefully the smallest) such that, if two solutions agree on it then they agree everywhere. Towards this goal, a  prototype class of Hamiltonian of the following form is studied carefully in Section \ref{sec:prototype} .
\begin{itemize}
\item[(H6)] Assume that
\[
H(x,r,p)=|p|^m - V(x) + f(r) \quad \text{ for } (x,r,p) \in \T^n \times \R \times \R^n.
\]
Here, $m \geq 1$ is a given number, and $V \in C(\T^n)$ is the potential energy with $\min_{\T^n} V=0$.
The function $f:\R \to \R$ is convex, and
\[
\begin{cases}
f(r)=0 \quad &\text{ for } r \leq 0,\\
f(r)>0 \quad &\text{ for } r>0.
\end{cases}
\]
\end{itemize}

\noindent Of course, we see that $f$ is not strictly increasing here, which makes the situation more interesting.
Here is our first result on the uniqueness property of (E) for the Hamiltonians in  the prototype class (H6) when $c>0$.

\begin{prop}\label{prop:unique-positive}
Assume {\rm (H6)}.
For $c>0$ fixed, {\rm (E)} admits a unique solution $(u_c,c) \in C(\T^n) \times (0,\infty)$.
\end{prop}
Next, we consider the case that $c=0$.
By using Proposition \ref{prop:unique-positive}, a priori estimates, and Arzel\`a-Ascoli's theorem,
we can easily show that under (H6), (E) has a solution $(u,0) \in C(\T^n) \times \R$ (see the last part of the proof of Proposition \ref{prop:prot} for a proof of this fact).

As $\min_{\T^n} V =0$, denote by
\[
M_V = \left\{ x\in \T^n\,:\, V(x) = \min_{\T^n} V = 0\right\}.
\]
Here is our second result along this line.
\begin{prop}\label{prop:unique-0}
Assume {\rm (H6)}.
Let $c=0$. 
Then, $M_V$ is a uniqueness set for {\rm (E)}, that is, if $(u_1,0), (u_2,0)$ are two solutions to {\rm (E)}, and $u_1=u_2$ on $M_V$, then $u_1=u_2$.
\end{prop}

\noindent It is worth noting that Proposition \ref{prop:unique-0} was first obtained in \cite{NR} when $f \equiv 0$.

\medskip

The uniqueness structure of solutions to (E), with Hamiltonians beyond the class of (H6), is more involved. To study it in a systematic way, we apply the nonlinear adjoint method and develop further the ideas in \cite{MT-P}. This is done in  Section \ref{sec:structure}.
We refer the readers to \cite{Ev1, T1, CGMT, MT-A, MT-P, LMT} and the references therein for the developments of the nonlinear adjoint method.

Under assumptions (H1), (H2b), (E) admits a solution $(v,c) \in C(\T^n) \times \R$.
As noted above, the constant $c$, that is the right hand side of (E), is not unique in general.
Therefore, to discuss the uniqueness structure of (E), we fix a $c\in \R$ such that (E) has a solution $v\in C(\T^n)$.
By a further normalization (setting $\tilde H(x,r,p) = H(x,r,p)-c$ for $(x,r,p) \in \T^n \times \R \times \R^n$), we may assume that $c=0$.
We hence study the uniqueness structure for the following problem
\begin{equation}\label{E-0}
H(x,u,Du)=0 \quad \text{ in } \T^n.
\end{equation}
Our main result on the uniqueness structure of \eqref{E-0} is as follows.

\begin{thm}\label{thm:unique}
Assume {\rm (H1), (H2b), (H3), (H4), (H5)}. Let $\cM$ be the set of measures in Definition \ref{def:M} of Section \ref{sec:structure}. Then, for any two solutions $u_1, u_2$ to \eqref{E-0}, the condition
\[
\int_{\T^n} u_1(x) \, d\nu(x) \leq \int_{\T^n} u_2(x)\, d\nu(x) \quad \text{ for all } \nu \in \cM
\]
implies $u_1 \leq u_2$. In particular, $M := \overline{ \bigcup_{\nu \in \cM} {\rm supp} (\nu)}$ is a uniqueness set for \eqref{E-0}.
\end{thm}
As described in Definition \ref{def:M}, $\cM$ contains adjoint measures associated to solutions of \eqref{E-0}. The whole construction of these measures is done in Section \ref{sec:structure}.
Theorem \ref{thm:unique} is a generalized version of \cite[Theorem 1.1]{MT-P}. 
See also a related work \cite{T}. 
To the best of our knowledge, Theorem \ref{thm:unique} is new in the literature.
In two specific situations (see Subsections \ref{subsec:structure1}--\ref{subsec:structure2}), we have clear understanding about this $\cM$.
In particular, we find a natural link between Theorem \ref{thm:unique} and Proposition \ref{prop:unique-0} above (see Proposition \ref{prop:prot}).

\subsection*{Organization of the paper}
The paper is organized as following.
In Section \ref{sec:exist}, we give the proofs of Theorems \ref{thm:exist1}--\ref{thm:exist2}.
Besides, we give three examples of Hamiltonians satisfying requirements of Theorem \ref{thm:exist1} in Subsection \ref{subsec:exist}.
In Section \ref{sec:ex}, we give various new examples to discuss nonuniqueness issues of solutions to (E).
Then, Section \ref{sec:prototype} is devoted to further analysis on a uniqueness set for a prototype case that was discussed in Section \ref{sec:ex}.
Finally, in Section \ref{sec:structure}, we use the nonlinear adjoint method to study systematically the uniqueness structure of solutions to (E).
Various connections with classical results and with the prototype case in Section \ref{sec:prototype} are discussed in deep too.


\section{Existence results for solutions to (E)} \label{sec:exist}
\subsection{Proof of Theorem \ref{thm:exist1}}
\begin{proof}[Proof of Theorem \ref{thm:exist1}]
The main idea is to use the Perron method to get the existence result.

Let $M= \|\psi\|_{L^\infty(\T^n)} + \|\varphi\|_{L^\infty(\T^n)} +1$.
By assumptions (H1) and (H2a), there exists $C_2>0$ such that
\[
H(x,r,p) \leq c \quad \text{ for some } (x,r) \in \T^n \times [-M, M] \quad \Longrightarrow \quad |p| \leq C_2.
\] 
Define, for $x\in \T^n$,
\begin{multline*}
u(x) = \sup \Big\{ v(x)\,:\, \psi \leq v \leq \varphi, \, \|Dv\|_{L^\infty(\T^n)} \leq C_2,\\
\text{and $v \in \Lip(\T^n)$ is a viscosity subsolution to (E)} \Big\}.
\end{multline*}
Of course, $u$   is well-defined as $\psi$ itself is an admissible subsolution in the above formula.
Furthermore, it is clear that $u$ is Lipschitz in $\T^n$, and $\|Du\|_{L^\infty(\T^n)} \leq C_2$.
By the stability of viscosity subsolutions, we have that $u$ is a viscosity subsolution to (E).

Hence, we only need to show that $u$ is a viscosity supersolution to (E).
Assume by contradiction that this is not the case. 
Then, there exist a smooth test function $\phi \in C^\infty(\T^n)$ and a point $x_0 \in \T^n$ such that
\[
\begin{cases}
u(x_0) = \phi(x_0),\
u(x) >\phi(x) \quad \text{ for all } x \in \T^n \setminus \{ x_0\},\\
H(x_0,u(x_0),D\phi(x_0)) = H(x_0,\phi(x_0),D\phi(x_0)) <c.
\end{cases}
\]
There are two cases to be considered here.
The first case is when $u(x_0)=\varphi(x_0)$. 
This means that $\phi$ touches  $\varphi$ from below at $x_0$.
By the definition of viscosity supersolutions, 
\[
H(x_0,\phi(x_0),D\phi(x_0))  \geq c,
\]
which implies a contradiction immediately.

The second case is when $u(x_0) <\varphi(x_0)$.
There exist $r,\ep>0$ sufficiently small such that
\[
\begin{cases}
u(x) < \varphi(x) -\ep \quad &\text{ for all } x \in B(x_0,r),\\
\phi(x) < u(x) -\ep \quad &\text{ for all } x\in \partial B(x_0,r),\\
H(x,\phi(x), D\phi(x)) < c - C_1 \ep \quad &\text{ for all } x \in B(x_0,r),\\
|D\phi(x)| \leq C_2  \quad &\text{ for all } x \in B(x_0,r).
\end{cases}
\]
Now, set
\[
\ol{u}(x) = 
\begin{cases}
\max\{u(x),\phi(x)+\ep\} \quad &\text{ for all } x\in B(x_0,r),\\
u(x) \quad &\text{ for all } x \in \T^n \setminus B(x_0,r).
\end{cases}
\]
It is quite clear that $\ol{u}$ is a viscosity subsolution to (E) thanks to (H1), and $\|D\ol{u}\|_{L^\infty(\T^n)} \leq C_2$. 
This again leads to a contradiction.
The proof is complete.

\end{proof}

\subsection{Examples of Hamiltonians satisfying Theorem \ref{thm:exist1}} \label{subsec:exist}

\begin{ex}[Classical setting with no $r$-dependence] \label{ex:1}
{\rm 
If $H(x,r,p)= \tilde H(x,p)$ with $\tilde H$ satisfies {\rm (H2a)}, that is,
\[
\lim_{|p| \to \infty}  \tilde H(x,p) = +\infty \quad \text{ uniformly for } x\in \T^n.
\]
Then we use classical results (see \cite{LPV} for example) to have the existence of a unique constant $c\in \R$ such that
\[
\tilde H(x,Dv) = c \quad \text{ in } \T^n
\]
has a viscosity solution $v\in C(\T^n)$.
Then, we can simply choose $\psi=\varphi=v$.
In this example, $c$ is unique.
}
\end{ex}

\begin{ex}[Strictly monotone setting] \label{ex:2}
{\rm
Assume {\rm (H1), (H2a)}.
If there exists $\al>0$ such that
\[
H_r(x,r,p) \geq \al \quad \text{ for all } (x,r,p) \in \T^n \times \R \times \R^n,
\]
then, for each fixed $c \in \R$, we can choose
\[
 \psi \equiv -\frac{1}{\al}\left(\|H(\cdot,0,0)\|_{L^\infty(\T^n)} +|c|\right), \quad \text{and} \quad \varphi \equiv \frac{1}{\al}\left(\|H(\cdot,0,0)\|_{L^\infty(\T^n)} +|c|\right).
\]
Therefore, (E) has solutions for each $c \in \R$. This is consistent with classical results (see \cite{CEL, CL} for example).
}
\end{ex}

\begin{ex}[Non-monotone setting] \label{ex:3}
{\rm
Assume {\rm (H1), (H2b)}.
Let us assume further that
\begin{equation}\label{nm-condition1}
\begin{cases}
\max_{x\in \T^n} H(x,0,0) =  H(x_0,0,0) \quad &\text{ for some } x_0 \in \T^n,\\
H(x_0,0,0) \leq H(x_0,0,p) \quad &\text{ for all } p \in \R^n.
\end{cases}
\end{equation}
Note that the requirements of this example are stronger than those in Theorem \ref{thm:exist2} (as we only assume (H1) and (H2b) there).
Nevertheless, this is a direct application of Theorem \ref{thm:exist1}, and hence, it is worth pointing it out here.
Examples of Hamiltonians satisfying (H1), (H2b), and \eqref{nm-condition1} are many.
A typical one is $H(x,r,p) = |p|^m + f(r) + V(x)$ for $(x,r,p) \in \T^n \times \R \times \R^n$.
Here, $m>1$ is fixed, $V\in C(\T^n)$, and $ f \in \Lip(\R)$ with Lipschitz constant at most $C_1$.
Of course, there is no requirement on convexity of $H$ in $p$ here.

\medskip

In this setting, we choose first
\[
c=\max_{x\in \T^n} H(x,0,0) =  H(x_0,0,0), \quad \text{and} \quad \psi \equiv 0.
\]
We now construct $\varphi$. Let $Q(x_0)= x_0 + [-1/2,1/2]^n$ be the unit cube centered at $x_0$. For $s>0$ sufficiently large, set
\[
\varphi(x) = s |x-x_0| \quad \text{ for all } x\in Q(x_0),
\]
and extend $\varphi$ to $\R^n$ periodically.
We claim that $\varphi$ is a supersolution to (E).

It is clear that $\varphi$ is not differentiable at $x_0$ and $\partial Q(x_0)$.
We do not have to worry about $\partial Q(x_0)$ as for any $x \in \partial Q(x_0)$, $D^- \varphi(x) = \emptyset$.
At $x=x_0$, we have $D^-\varphi(x_0) = \ol{B}(0,s)$. By the second line in assumption \eqref{nm-condition1}, we have that
\[
H(x_0,0,p) \geq H(x_0,0,0) = c \quad \text{ for all } p \in \ol{B}(0,s).
\]
For other values of $x$, $\varphi$ is smooth and $|D\varphi(x)|=s$. We use (H1) and (H2b), the superlinearity of $H$, to get
\[
H(x,\varphi(x),D\varphi(x)) \geq H(x,0,D\varphi(x)) - C_1s \geq c,
\]
for $s$ sufficiently large.}
\end{ex}

\subsection{Proof of Theorem \ref{thm:exist2}}
We always assume (H1), (H2b) in this subsection.
We first formulate Theorem \ref{thm:exist2} as a fixed point problem by adding a monotone term to (E).
 
 Fix $\lambda > C_1+1$. 
 For each $u \in C(\T^n)$, let $v \in \Lip(\T^n)$  be the unique viscosity solution to
\begin{equation}\label{eq:u-v}
\lambda v + H(x,u,Dv) - \lambda u = 0 \quad \text{ in } \T^n.
\end{equation}
Note that we use the Perron method to get directly a solution $v \in \Lip(\T^n)$.
Then, uniqueness of \eqref{eq:u-v} follows immediately.

Denote by $G(u) = w := v - \min_{\mathbb{T}^n} v$. It is clear that $w=G(u)$ solves
\begin{equation}\label{eq:G}
\lambda (w-u) + H(x,u,Dw) = -\lambda \min_{\mathbb{T}^n} v \quad \text{ in } \T^n.
\end{equation}
Our aim now is to show that the map $G: C(\T^n) \to C(\T^n)$ has a fixed point by using the Schauder fixed point theorem. We first show that $G$ is continuous.

\begin{lem}\label{lem:G1}
For every $u_1, u_2 \in C(\T^n)$,
\[
\|G(u_1) - G(u_2)\|_{L^\infty(\T^n)} \leq 4 \|u_1-u_2\|_{L^\infty(\T^n)}.
\]
\end{lem}

\begin{proof}
For $i=1,2$, let $v_i \in C(\T^n)$ be the unique viscosity solution to
\begin{equation*}
\lambda v_i + H(x,u_i,Dv_i) - \lambda u_i = 0  \quad \text{ in } \T^n.
\end{equation*}
We use (H1) to deduce that $v_1$ is a subsolution to
\[
\lam v_1 + H(x,u_2, Dv_1) - \lam u_2 \leq ( \lam + C_1) \|u_1-u_2\|_{L^\infty(\T^n)} \quad \text{ in } \T^n.
\]
By the comparison principle, we yield
\[
v_1 - \left(1 + \frac{C_1}{\lam} \right) \|u_1-u_2\|_{L^\infty(\T^n)} \leq v_2.
\]
By the same argument, we obtain
\[
\|v_1-v_2\|_{L^\infty(\T^n)} \leq  \left(1 + \frac{C_1}{\lam} \right) \|u_1-u_2\|_{L^\infty(\T^n)} \leq 2  \|u_1-u_2\|_{L^\infty(\T^n)}.
\]
Next, for $i=1,2$, denote by 
\[
w_i = G(u_i) = v_i - \min_{\T^n} v_i = v_i - v_i(x_i) \quad \text{ for some } x_i \in \T^n.
\]
Then, for any $x\in \T^n$,
\begin{align*}
w_1(x) - w_2(x) &= (v_1(x) - v_1(x_1)) - (v_2(x) - v_2(x_2))\\
&= (v_1(x) - v_2(x)) + (v_2(x_2) - v_1(x_1)) \\
&=  (v_1(x) - v_2(x)) + (\min_{\T^n} v_2 - v_1(x_1))\\
&\leq (v_1(x) - v_2(x)) + (v_2(x_1) - v_1(x_1)) \leq 2 \|v_1-v_2\|_{L^\infty(\T^n)} \\
&\leq 4\|u_1-u_2\|_{L^\infty(\T^n)}.
\end{align*}
By a symmetric argument, the proof is complete.
\end{proof}

Set $C_0 = \max_{x\in \T^n} |H(x,0,0)|$.
By (H2b), we pick $\al>0$ such that, if $|p| \geq \al$, then
\[
H(x,0,p) \geq 3\lam (C_0 + \al (1+\sqrt{n})).
\]
Denote by
\[
K := \{u \in \Lip(\mathbb{T}^n) \,:\, u \geq 0, \|u\|_{L^\infty(\T^n)} +\|Du\|_{L^\infty(\T^n)}  \le \al (1+\sqrt{n})\}.
\]
Clearly, $K$ is a non-empty convex and compact subset of $C(\T^n)$.

\begin{lem}\label{lem:G2}
We have that $G(K) \subset K$, where $G(K):=\{G(v)\,:\, v\in K\}$. 
\end{lem}

\begin{proof}
Fix $u \in K$, and let $v \in \Lip(\T^n)$ be the viscosity solution to \eqref{eq:u-v}.
First of all, it is clear that $C_0 + 2\al (1+\sqrt{n})$ and $-C_0 - \al(1+\sqrt{n})$ are, respectively, a supersolution and a subsolution to \eqref{eq:u-v}.
The comparison principle then gives
\[
-C_0 - \al(1+\sqrt{n}) \leq v \leq C_0 + 2\al (1+\sqrt{n}) \quad \text{ in } \T^n.
\]
Thus, for a.e. $x\in \T^n$,
\[
H(x,0,Dv(x)) \leq \lam (u(x) - v(x)) + C_1 \|u\|_{L^\infty(\T^n)} < 3\lam (C_0 + \al (1+\sqrt{n})),
\]
which, together with the choice of $\al$, yields $\|Dv\|_{L^\infty(\T^n)} \leq \al$.

Hence, for $w=v- \min_{\T^n} v$, we have $w \in K$. 
\end{proof}

\begin{proof}[Proof of Theorem \ref{thm:exist2}]
By Lemmas \ref{lem:G1}--\ref{lem:G2}, we are able to apply Schauder's fixed point theorem to imply the existence of  $u \in K$ such that
\begin{equation*}
G(u) = u. 
\end{equation*}
This means that, for $v \in C(\T^n)$ solves \eqref{eq:u-v}, $u= v - \min_{\T^n}v $ satisfies
\begin{equation*}
H(x,u,Du) = c:= -\lambda \min_{\T^n} v \quad \text{ in } \T^n.
\end{equation*}
\end{proof}


\section{Some examples on nonuniqueness of solutions to (E)} \label{sec:ex}

In this section, we give several examples to illustrate the nonuniqueness of solutions to (E). 
Our main guiding principle here is that, if $r \mapsto H(x,r,p)$ is not strictly monotone for each $(x,p) \in \T^n \times \R^n$,
then it is highly unlikely the case that (E) has a unique solution.

\begin{ex}\label{ex-4}
{\rm
Assume that $n=1$, and 
\[
H(x,r,p) = |p|^2 + V(x) - \lambda r \quad \text{ for } (x,r,p) \in \T \times \R \times \R,
\]
where $\lambda > 2$ is given.
Clearly, $H_r(x,r,p) = -\lam <0$.
Here, the potential energy $V$ is defined as
\begin{equation*}
V(x) = \begin{cases}
\frac{1}{4} (x-\frac{1}{2})^2 \quad& 0 \leq x  \leq \frac{1}{2},\\
\frac{1}{4} (x+\frac{1}{2})^2 \quad& -\frac{1}{2} \leq x \leq 0.
\end{cases}
\end{equation*}
Extend $V$ to $\R$ in a periodic way.
It is worth noting that $V$ is $C^1$ on the torus except at $0$,
and $V$ is a viscosity solution to 
\[
|V'|^2 - V = 0 \quad \text{ in } \T.
\] 
We use this fact to imply that
\begin{equation*}
u_1  = \frac{\lambda + \sqrt{\lambda^2 - 4}}{2} V \quad \text{and} \quad u_2 = \frac{\lambda - \sqrt{\lambda^2 - 4}}{2} V
\end{equation*}
are two different viscosity solutions of the equation
\begin{equation*}
-\lambda u +  |u'|^2 + V = 0 \quad \text{ in } \T.
\end{equation*}
In other words, $(u_1,0)$ and $(u_2,0)$ are two pairs of solutions to (E) with $c=0$ here.

It is also clear that (E) has at least two solutions for every $c\in \R$. Indeed, for each $c\in \R$, and $i=1,2$, define
\[
u_{i,c} = u_i -\frac{c}{\lam}.
\]
Then $(u_{i,c},c)$ is a solution to (E) for $i=1,2$.
}
\end{ex}
See also \cite[Section 1.4]{G2} for similar comments on the nonuniqueness of both $c$ and $u$.
Surely, one objection that one may have for the above example is that $H_r(x,r,p) = -\lam <0$, which is too restrictive.
Nevertheless, in the following example, we will show that nonuniqueness appears even when $H_r(x,r,p) \geq 0$.

\begin{ex}\label{ex-5}
{\rm
Assume that 
\[
H(x,r,p) = |p| - V(x) + f(r) \quad \text{ for } (x,r,p) \in \T^n \times \R \times \R^n.
\]
Here, $f:\R \to \R$ is defined as
\[
f(r)=
\begin{cases}
0 \quad &\text{ for } r \leq 0,\\
r \quad &\text{ for } r>0.
\end{cases}
\]
And $V \in C(\T^n)$ is the potential energy with $\min_{\T^n} V=0$.
Let $w \in C(\T^n)$ be the viscosity solution to
\begin{equation}\label{eq:ex5}
w + |Dw| - V = 0 \quad \text{ in } \T^n.
\end{equation}
As $0$ is a subsolution to the above, $w \geq 0$.
Besides, it is clear that $w \leq V$, which gives us that $\{V=0\} \subset \{w=0\}$.
In particular, $f(w)=w$ always, and hence, $(w,0)$ is a solution to (E).
From this, it is also clear that (E) has a solution $(w+c,c)$ for every $c\geq 0$. 

Let us now proceed to describe more solutions to (E) with $c=0$.
Consider the usual ergodic (cell problem)
\begin{equation}\label{eq:ex5-E}
|Dv|-V = 0 \quad \text{ in } \T^n,
\end{equation}
which is of eikonal type.
For each solution $v \in C(\T^n)$ of \eqref{eq:ex5-E}, take $C> \|v\|_{L^\infty(\T^n)}$, then $v-C$ is still a solution to \eqref{eq:ex5-E}, and $v-C \leq 0$.
Thus, $f(v-C)=0$, and $(v-C,0)$ is a solution to (E).
}
\end{ex}

\begin{ex}\label{ex-6}
{\rm
Let us analyze further  Example \ref{ex-5}.
Basically, if we put more structural condition on $V$, we are able to find more nontrivial solutions to (E) with $c=0$. Below, we use the setting in Example \ref{ex-5} and present an example in which the solution $u$ has range in both branches of the $f$ function. 

We assume further that $V \in C^1(\T^n)$, and that for some $r \in (0,\frac12)$ we have
\begin{equation}\label{V-zero}
\begin{cases}
V \ge 0 \text{ \ in \ } \T^n \;\; \text{ and } \;\; \{V=0\} = \{0\} \cup \partial B(0,r),  \\
V(x) = \tilde V(|x|) \;\; \text{ for all } |x| \leq r.
\end{cases}
\end{equation}
Here, $\tilde V: [0,r] \to \R$ is $C^1$, $\tilde V \geq 0$, and $\{\tilde V=0\} = \{0,r\}$.

Let $w$ be the solution to \eqref{eq:ex5}. Then clearly $0 \leq w \leq V$, and 
\[
w(x)=0, \quad \text{and} \quad Dw(x)=0 \quad \text{ for each } x \in \partial B(0,r).
\]
Moreover, $w$ is not constantly zero in $\T^n\setminus \ol{B}(0,r)$. 

Next, we construct $\phi:[0,r] \to \R$ such that
\[
\begin{cases}
\phi'(s) = \tilde V(s) \quad \text{ for all } 0 \leq s \leq r,\\
\phi(r)=0.
\end{cases}
\]
Then define $u:\T^n \to \R$ by
\[
u(x)=
\begin{cases}
\phi(|x|) \quad &\text{ for } x \in B(0,r),\\
w(x) \quad &\text{ for } x \in \T^n \setminus B(0,r).
\end{cases}
\]
Clearly, $u <0$ in $B(0,r)$, and $Du(x)=0$ on $\partial B(0,r)$.
Besides, $Du(0)=0$, and therefore, $u$ solves
\[
|Du(x)| = \phi'(|x|) = \tilde V(|x|) = V(x) \quad \text{ for } x\in B(0,r).
\]
We conclude that $(u,0)$ is a solution to (E).
}
\end{ex}

Finally, let us consider the following example, where the Hamiltonian is of magnetic type.

\begin{ex}\label{ex-7}
{\rm
Assume that 
\[
H(x,r,p) = |p|^2 -  p\cdot D\varphi(x)  + f(r) \quad \text{ for } (x,r,p) \in \T^n \times \R \times \R^n.
\]
Here, $f:\R \to \R$ is defined as
\[
f(r)=
\begin{cases}
0 \quad &\text{ for } r \leq 0,\\
r \quad &\text{ for } r>0.
\end{cases}
\]
And $\varphi \in C^1(\T^n)$ is given.
Let us now proceed to describe various solutions to (E) with $c=0$.
The corresponding equation reads
\begin{equation}\label{eq:mag1}
|Du|^2 - Du\cdot D\varphi + f(u) =0 \quad \text{ in } \T^n.
\end{equation}
It is clear that $u \equiv 0$ is a trivial solution to the above.

Now, take any solution $u \in C(\T^n)$ of \eqref{eq:mag1}.
We show that $u \leq 0$. Indeed, take $x_1 \in \T^n$ so that $u(x_1) = \max_{\T^n} u$.
By the viscosity subsolution test, we deduce that
\[
f(u(x_1)) \leq 0 \quad \Rightarrow \quad u(x_1) \leq 0.
\]
Thus, $u\leq 0$, and $u$ solves a usual ergodic (cell problem) without $f$ as following
\begin{equation}\label{eq:mag2}
|Du|^2 - Du\cdot D\varphi = 0 \quad \text{ in } \T^n,
\end{equation}
which is quite an interesting phenomenon.
It is clear that $u_1 \equiv C_1$ for any constant $C_1 \leq 0$, and $u_2 \equiv \varphi+C_2$ for any constant $C_2 \leq -\|\varphi\|_{L^\infty(\T^n)}$ are solutions to \eqref{eq:mag1} and \eqref{eq:mag2}.
Besides, by stability results for convex Hamiltonian, we have further that
\[
u_3 = \min\{u_1,u_2\} = \min\{C_1, \varphi+C_2\}
\]
is also a solution to \eqref{eq:mag1} and \eqref{eq:mag2}.
See also \cite[Example 6.2]{LMT}.
Note that we do not claim here that we have described all solutions to \eqref{eq:mag1}.
}
\end{ex}


\section{Further analysis on a uniqueness set for a prototype case} \label{sec:prototype}

Let us now come back to Hamiltonians of type in Example \ref{ex-5} to do further analysis.
In this section, we always assume (H6). That is, we consider a general class of Hamiltonian of the form
\[
H(x,r,p)=|p|^m - V(x) + f(r) \quad \text{ for } (x,r,p) \in \T^n \times \R \times \R^n.
\]
Here, $m \geq 1$ is a given number, and $V \in C(\T^n)$ is the potential energy with $\min_{\T^n} V=0$.
The function $f:\R \to \R$ is convex, and
\[
\begin{cases}
f(r)=0 \quad &\text{ for } r \leq 0,\\
f(r)>0 \quad &\text{ for } r>0.
\end{cases}
\]
It is clear that the Hamiltonian in Example \ref{ex-5} is a specific case of this class.
Our goal here is to analyze more about solutions of (E) for fixed $c\geq 0$.
We first show that $f$ is nondecreasing.

\begin{lem}
Let $f \in C(\R)$ be given as above. Then $f$ is nondecreasing.
\end{lem}

\begin{proof}
Take $0<r<s$. By the convexity of $f$, we have
\[
0<f(r) \leq \frac{r}{s} f(s) + \left(1-\frac{r}{s}\right) f(0) = \frac{r}{s} f(s) \leq f(s).
\]
\end{proof}

We give a proof of our first uniqueness result when $c>0$.

\begin{proof}[Proof of Proposition \ref{prop:unique-positive}]
As we explain in Section \ref{sec:exist}, Example \ref{ex-5}, for every $c\ge0$, (E) has viscosity solutions. 
Let $u \in C(\T^n)$ be a solution to (E) with the given $c>0$ on the right hand side, that is,
\begin{equation}\label{E-c}
|Du|^m - V(x) + f(u) = c \quad \text{ in } \T^n.
\end{equation}
Then, $f(u) \leq V+c$, which means that $u \leq C$.

Next, pick $x_1 \in \T^n$ so that $u(x_1) = \min_{\T^n} u$. By the viscosity supersolution test,
\[
-V(x_1)+ f(u(x_1)) \geq c \quad \Rightarrow  \quad f(u(x_1)) \geq c >0 \quad \Rightarrow \quad u(x_1) \geq \bar c = f^{-1}(c)>0.
\]
Therefore, $ \bar c \leq u \leq C$. Since $f$ is convex and increasing, we can find $0<\lam \leq \Lam$ such that
\[
\lam \leq f'(r) \leq \Lam \quad \text{ for a.e. } r \in [\bar c, C].
\]
We now can apply classical theory of viscosity solution to imply the uniqueness of solutions to \eqref{E-c}.
For convenience later on, denote by $u_c$ this unique solution.
\end{proof}
One key feature we used in the above proof is that $\phi \equiv 0$ is a subsolution to (E) for all $c\geq 0$.
In particular, for $c>0$, $\phi \equiv 0$ is a strict subsolution, and therefore, we were able to get that $u>0$.

On the other hand, for $c=0$, we have seen in Examples \ref{ex-5} and \ref{ex-6} that we do not have uniqueness for (E).
It turns out that $M_V$ is a uniqueness set for (E) in this case, which is exactly the content of the following proof.

\begin{proof}[Proof of Proposition \ref{prop:unique-0}]
Assume that $(u_1,0), (u_2,0)$ are two solutions to {\rm (E)}, and $u_1=u_2$ on $M_V$. If $M_V = \T^n$, there is nothing to prove. Hence, we assume below that $\T^n\setminus M_V$ is nonempty.

Assume by contradiction  that there exists $x_0\in\T^n\setminus M_V$ such that
\[
\max_{\T^n} (u_1-u_2) = u_1(x_0) - u_2(x_0) >0.
\] 
Take $\lam \in (0,1)$, which is very close to $1$, such that $\lam u_1(x_0) > u_2(x_0)$, and
\[
\lam u_1(x_0) - u_2(x_0) > \lam u_1(x) - u_2(x) = -(1-\lam)u_1(x) \quad \text{ for all } x\in M_V.
\] 
Then, 
\[
\max_{\T^n}(\lam u_1-u_2)=(\lam u_1-u_2)(x_\lam)>0 
\] 
for some $x_\lam\in\T^n\setminus M_V$. 

Due to the convexity of $r \mapsto f(r)$ and $p \mapsto |p|^m$, denote by  $v=\lam u_1= (1-\lam)0+\lam u_1$. 
Then, $v$ satisfies 
\[
|Dv|^m -\lam V+ f(v) \leq 0 \quad \text{ in } \T^n.
\]
We now perform the usual doubling variables technique. For $\ep>0$, consider the auxiliary function
\[
\Phi^\ep(x,y) = v(x) - u_2(y)  - \frac{|x-y|^2}{2\ep}.
\]
Then, $\Phi^\ep$ admits a maximum at $(x_\ep, y_\ep)$, and by passing to a subsequence if needed, $(x_\ep, y_\ep) \to (x_\lam, x_\lam)$ as $\ep \to 0$.
By the viscosity solution tests, we have
\[
\left|\frac{x_\ep-y_\ep}{\ep} \right|^m  - \lam V(x_\ep) + f(v(x_\ep)) \leq 0,
\]
and
\[
\left|\frac{x_\ep-y_\ep}{\ep} \right|^m  -  V(y_\ep) + f(u_2(y_\ep)) \geq 0.
\]
Combine the two inequalities above to yield
\[
-\lam V(x_\ep) + f(v(x_\ep)) \leq -  V(y_\ep) + f(u_2(y_\ep)).
\]
Then, let $\ep \to 0$ to get further that
\[
-\lam V(x_\lam) + f(v(x_\lam)) \leq  - V(x_\lam) + f(u_2(x_\lam)).
\]
Since, $v(x_\lam) > u_2(x_\lam)$, $f(v(x_\lam)) \geq f(u_2(x_\lam))$.
Thus, we end up with a contradiction as $V(x_\lam)>0$.
The proof is complete.
\end{proof}


\section{Uniqueness structure of solutions to (E)} \label{sec:structure}
In this section, we always assume (H1), (H2b), (H3), (H4), and (H5).
Recall that after normalization as explained in Introduction, 
we assume further that $c=0$, and the ergodic problem becomes
\begin{equation*}
H(x,u,Du)=0 \quad \text{ in } \T^n.
\end{equation*}
In \cite{WWY-1}, the authors put an admissible condition (see \cite[Assumption (A), Theorem 1.1]{WWY-1}) to guarantee that (E) has a viscosity solution with $c=0$.
Note further that  in \cite{WWY-1}, they need to require a stronger condition, that is, $H_r(x,r,p)>0$,  than our (H4), which basically guarantees the uniqueness of viscosity solutions to (E).
See \cite{GMT, CCIZ} for related works.

We now use the nonlinear adjoint method to study \eqref{E-0}.

\subsection{Preliminaries}
Here is a first preparatory lemma. Since this is elementary, we omit the proof. 
\begin{lem}\label{lem:conv}
Let $u \in \Lip(\T^n)$ be a solution to \eqref{E-0}.
Let $\rho \in C_c^\infty(\R^n,[0,\infty))$ be a standard mollifier.
For $\del>0$, let $\rho^\del(x) = \del^{-n} \rho(\del^{-1}x)$ for all $x\in \R^n$. 
Denote by $u^\del = \rho^\del * u$. Then,
\[
\|u^\del - u\|_{L^\infty(\T^n)} \leq C\del,
\]
and
\[
\|Du^\del\|_{L^\infty(\T^n)} + \del \|D^2 u^\del\|_{L^\infty(\T^n)} \leq C.
\]
\end{lem}
Let us consider the following Cauchy problems
\begin{equation}\label{eq:C1}
\begin{cases}
\ep w^\ep_t + H(x,w^\ep,Dw^\ep) = \ep^4 \Del w^\ep \quad &\text{ in } \T^n \times (0,1),\\
w^\ep(x,0) = u^{\ep^4}(x) \quad &\text{ on } \T^n,
\end{cases}
\end{equation}
and
\begin{equation}\label{eq:C2}
\begin{cases}
\ep \phi^\ep_t + H(x,\phi^\ep,D\phi^\ep) = \ep^4 \Del \phi^\ep \quad &\text{ in } \T^n \times (0,1),\\
\phi^\ep(x,0) = u(x) \quad &\text{ on } \T^n.
\end{cases}
\end{equation}
Here, $u^{\ep^4}$ is $u^\del$ with $\del=\ep^4$.
\begin{lem}\label{lem:close}
We have
\[
\|w^\ep - \phi^\ep\|_{L^\infty(\T^n \times [0,1])} \leq C\ep^4.
\]
\end{lem}

\begin{proof}
Recall that $\|u^{\ep^4} - u\|_{L^\infty(\T^n)} \leq C\ep^4$. 

Let $\varphi(x,t) = w^\ep(x,t) + C\ep^4$ for $(x,t) \in \T^n \times [0,1]$, then $\varphi(x,0) \geq \phi^\ep(x,0)$ for $x\in \T^n$.
Besides, thanks to (H4),
\begin{align*}
&\ep \varphi_t + H(x,\varphi, D\varphi) - \ep^4 \Del \varphi\\
=\,&  \ep w^\ep_t + H(x,w^\ep+ C\ep^4,Dw^\ep) - \ep^4 \Del w^\ep
\geq  \ep w^\ep_t + H(x,w^\ep ,Dw^\ep) - \ep^4 \Del w^\ep =0,
\end{align*}
which means that $\varphi$ is a supersolution of \eqref{eq:C2}.
By the comparison principle, $\phi^\ep \leq \varphi$, and thus,
\[
\phi^\ep(x,t) \leq w^\ep(x,t) + C\ep^4 \quad \text{ for all } (x,t) \in \T^n \times [0,1].
\]
By a symmetric argument, the proof is complete.
\end{proof}

The next result concerns gradient bound of $w^\ep$.

\begin{lem}\label{lem:grad-bound}
There is a constant $C>0$ independent of $\ep>0$ such that
\[
\ep \|w^\ep_t\|_{L^\infty(\T^n \times [0,1])} + \|Dw^\ep\|_{L^\infty(\T^n \times [0,1])} \leq C.
\]
\end{lem}

\begin{proof}
Denote by
\[
\varphi^{\pm}(x,t) = w^\ep(x,0) \pm \frac{C}{\ep}t \quad \text{ for all } (x,t) \in \T^n \times [0,1].
\]
Then, $\varphi^-, \varphi^+$ are, respectively, a subsolution, and a supersolution to \eqref{eq:C1}, thanks to (H4).
Hence, by the comparison principle,
\[
\varphi^- \leq w^\ep \leq \varphi^+ \quad \Rightarrow \quad\|w^\ep(\cdot,s)-w^\ep(\cdot,0)\|_{L^\infty} \leq \frac{Cs}{\ep}.
\]
Note next that both $w^\ep$ and $w^\ep(\cdot, \cdot+s)$ solve \eqref{eq:C1} with initial data $w^\ep(\cdot,0)$ and $w^\ep(\cdot,s)$, respectively.
By the comparison principle,
\[
\|w^\ep(\cdot, \cdot+s)-w^\ep\|_{L^\infty} \leq \|w^\ep(\cdot,s)-w^\ep(\cdot,0)\|_{L^\infty} \leq \frac{Cs}{\ep}
 \quad \Rightarrow \quad \ep \|w^\ep_t\|_{L^\infty(\T^n)} \leq C.
\]
To prove the spatial gradient bound, we use the usual Bernstein method.
Let $\psi(x,t) = \frac{|Dw^\ep|^2}{2}$. Then $\psi$ satisfies
\[
\ep \psi_t  + D_pH\cdot D\psi + 2 H_r \psi + D_xH\cdot Dw^\ep = \ep^4 \Del \psi  - \ep^4 |D^2 w^\ep|^2.
\]
Assume that $\max_{\T^n \times [0,1]} \psi = \psi(x_0,t_0)$. If $t_0=0$, then we are done.
If $t_0 >0$, then by the maximum principle, noting that $H_r \geq 0$,
\[
D_xH \cdot Dw^\ep  + \ep^4 |D^2 w^\ep|^2 \leq 0 \quad \text{ at } (x_0,t_0).
\]
For $\ep<n^{-1}$, we have
\[
\ep^4 |D^2 w^\ep|^2 \geq (\ep^4 \Del w^\ep)^2  = (\ep w^\ep_t + H(x,w^\ep, Dw^\ep))^2 \geq \frac{1}{2} H(x,w^\ep, Dw^\ep)^2 - C.
\]
Therefore,
\[
\frac{1}{2} H(x,w^\ep, Dw^\ep)^2 +D_xH \cdot Dw^\ep  \leq C \quad \text{ at } (x_0,t_0),
\]
which, together with (H3), yields the desired result.
\end{proof}

\begin{lem}\label{lem:ep}
We have
\[
\|w^\ep - u\|_{L^\infty(\T^n \times [0,1])}+\|\phi^\ep - u\|_{L^\infty(\T^n \times [0,1])} \leq C\ep.
\]
\end{lem}
The proof of this is similar to that of \cite[Proposition 5.5]{LMT}.
Nevertheless, let us present a simple proof here for completeness.

\begin{proof}
We only need to show that $\|w^\ep - u\|_{L^\infty(\T^n \times [0,1])} \leq C\ep$.
Let us first get an upper bound for $w^\ep - u$.
Define an auxiliary function
\[
\Phi(x,y,t) = w^\ep(x,t) - u(y) - \frac{|x-y|^2}{2 \ep^2} - K\ep t \quad \text{ for } (x,y,t) \in \T^n \times \T^n \times [0,1],
\]
where $K>0$ is to be chosen.
Pick $(x_\ep, y_\ep, t_\ep ) \in  \T^n \times \T^n \times [0,1]$ so that
\[
\Phi(x_\ep, y_\ep, t_\ep ) = \max_{ \T^n \times \T^n \times [0,1]} \Phi.
\]
If $\Phi(x_\ep, y_\ep, t_\ep )  \leq 0$, then we are done as
\[
w^\ep(x,t) - u(x) = \Phi(x,x,t) +K \ep t \leq K \ep.
\]
Therefore, we can assume $\Phi(x_\ep, y_\ep, t_\ep ) > 0$.
This gives that $w^\ep(x_\ep,t_\ep)>u(y_\ep)$.

Let us consider first the case that $t_\ep>0$.
Since $w^\ep$ and $u$ are Lipschitz, by comparing $\Phi(x_\ep,y_\ep,t_\ep)$ with $\Phi(y_\ep,y_\ep,t_\ep)$, we deduce first that
\[
|x_\ep - y_\ep | \leq C\ep^2.
\]
By the viscosity subsolution and supersolution tests, we have
\[
K \ep^2  + H\left(x_\ep, w^\ep(x_\ep, t_\ep), \frac{x_\ep - y_\ep}{\ep^2} \right) \leq \ep^4 \frac{n}{\ep^2} = n\ep^2,
\]
and
\[
H\left(y_\ep, u(y_\ep), \frac{x_\ep - y_\ep}{\ep^2} \right) \geq 0.
\]
Combine these two inequalities, and use (H3), (H4) to imply
\begin{align*}
K \ep^2 &\leq  n\ep^2+   H\left(y_\ep, u(y_\ep), \frac{x_\ep - y_\ep}{\ep^2} \right) - H\left(x_\ep, w^\ep(x_\ep, t_\ep), \frac{x_\ep - y_\ep}{\ep^2} \right)\\
&\leq n \ep^2  + C|y_\ep - x_\ep| +H\left(x_\ep, u(y_\ep), \frac{x_\ep - y_\ep}{\ep^2} \right) - H\left(x_\ep, w^\ep(x_\ep, t_\ep), \frac{x_\ep - y_\ep}{\ep^2} \right)\\
& \leq n \ep^2  + C|y_\ep - x_\ep|  \leq (C+n) \ep^2.
\end{align*}
By picking $K = C+n+1$, we conclude that $t_\ep$ cannot be positive. 
Thus, $t_\ep  = 0$, and 
\[
\Phi(x_\ep, y_\ep, t_\ep ) \leq  u^{\ep^4}(x_\ep) -u(y_\ep) \leq C\ep^4 + C|x_\ep - y_\ep| \leq C\ep^2.
\]
Then, for $(x,t) \in \T^n \times [0,1]$,
\[
w^\ep(x,t) - u(x) = \Phi(x,x,t) +K \ep t \leq C\ep^2+ K \ep \leq C\ep.
\]

To get the other bound, we need to get an upper bound of $u-w^\ep$.
This can be done analogously to the above by carefully considering another auxiliary function
\[
\Psi(x,y,t) = u(x) - w^\ep(y,t) - \frac{|x-y|^2}{2 \ep^2} - K\ep t \quad \text{ for } (x,y,t) \in \T^n \times \T^n \times [0,1],
\]
where $K>0$ is to be chosen. We omit the proof of this part here.
\end{proof}

\subsection{Nonlinear adjoint method and adjoint measures}
Let $u \in \Lip(\T^n)$ be a solution to \eqref{E-0}.
Let $w^\ep$ be the solution to \eqref{eq:C1} with this fixed $u$, that is, $w^\ep(x,0) = u^{\ep^4}(x)$ in $\T^n$.

The linearized operator of \eqref{eq:C1} about the solution $w^\ep$ is
\[
\cL^\ep[\phi]= \ep \phi_t + H_r(x,w^\ep,Dw^\ep) \phi + D_pH(x,w^\ep,Dw^\ep)\cdot D\phi - \ep^4 \Del \phi.
\]
The corresponding adjoint equation is
\begin{equation}\label{eq:sig}
\begin{cases}
-\ep \sig^\ep_t +H_r(x,w^\ep,Dw^\ep) \sig^\ep - {\rm div}(D_pH(x,w^\ep,Dw^\ep) \sig^\ep)  = \ep^4 \Del \sig^\ep \quad &\text{in } \T^n \times (0,1),\\
\sig^\ep(x,1) = \del_{x_0}.
\end{cases}
\end{equation}
Here, $\del_{x_0}$ is the Dirac delta measure at $x_0 \in \T^n$.
It is clear that $\sig^\ep>0$ in $\T^n \times (0,1)$.

\begin{prop}\label{prop:sig}
The following holds
\begin{align*}
& \frac{d}{dt}\int_{\T^n}\sig^\ep(x,t)\,dx
= \frac{1}{\ep}\int_{\T^n}H_r(x,w^\ep,Dw^\ep)\sig^\ep\,dx\ge0, \\
& 0\le \int_{\T^n}\sig^\ep(x,t)\,dx\le 1 \quad\text{ for all } \quad 0\le t < 1. 
\end{align*}
\end{prop}

\begin{proof}
For $t\in (0,1)$, integrate \eqref{eq:sig} on $\T^n$ to yield 
\begin{align*}
\ep\dfrac{d}{dt}\int_{\T^n}\sig^\ep\,dx
&\,
=\int_{\T^n} H_r(x,w^\ep,Dw^\ep)\sig^\ep-\Div(D_pH(x,w^\ep,Dw^\ep)\sig^\ep)
-\ep^4\Del \sig^\ep\,dx\\
&\, 
=\int_{\T^n}H_r(x,w^\ep,Dw^\ep)\sig^\ep\,dx\ge0,
\end{align*}
which gives the first claim. The second claim follows immediately.
\end{proof}

For each $\sig^\ep$, there exist  a nonnegative Radon measure  $\nu^\ep \in\cR(\T^n )$ satisfying 
\[
\int_{0}^{1}\int_{\T^n}\psi(x)\sig^\ep(x,t)\,dxdt
=\int_{\T^n}\psi(x)\,d\nu^\ep(x) \quad \text{ for all } \psi \in C(\T^n).
\]
In fact, for a Borel measurable set $A\subset \T^n$,
\[
\nu^\ep(A) = \int_0^1 \int_A \sig^\ep(x,t)\,dxdt.
\]
By Proposition \ref{prop:sig}, $\nu^\ep(\T^n) \leq 1$.
We are able to pick a subsequence $\{\ep_j\} \to 0$ such that 
\[
\nu^{\ep_j}\rightharpoonup \nu
\]
as $j\to\infty$ weakly in the sense of measures. 

\begin{defn}\label{def:M}
We define the set $\cM\subset\cR(\T^n)$ as 
\[
\cM:=\bigcup_{\substack{u \in \cS, \, x_0\in\T^n\\ \{\ep_j\} \to 0}}\{\nu\}, 
\]
where $\cS$ denote the family of all viscosity solutions to \eqref{E-0}.  
Here, we collect all possible subsequential weak limits (in the sense of measure) of $\{\nu^\ep\}$ for all $x_0 \in \T^n$.

We call each measure $\nu \in \cM$ an adjoint measure of \eqref{E-0}.
We say that $\cM$ is the set of adjoint measures corresponding to \eqref{E-0}.
\end{defn}
\begin{rem}
It is important noting that the set $\cM$ is defined implicitly as 
it depends on all solutions to \eqref{E-0}, which are not known a priori. 
\end{rem}
It turns out that the adjoint measures give us the uniqueness property of solutions to \eqref{E-0} as stated in Theorem \ref{thm:unique}.
Here is the proof of our main theorem on uniqueness property.
\begin{proof}[Proof of Theorem \ref{thm:unique}]
For $i=1,2$, let $w^\ep_i$ be the solution to \eqref{eq:C1} with initial data $u_i^{\ep^4}$.
By the convexity assumption (H5), we subtract the equations for $w^\ep_1$ and $w^\ep_2$ to get
\begin{multline}\label{eq:diff}
\ep (w^\ep_1 - w^\ep_2)_t + H_r(x,w^\ep_2, Dw^\ep_2)(w^\ep_1 - w^\ep_2)+ D_pH(x,w^\ep_2, Dw^\ep_2)\cdot D(w^\ep_1 - w^\ep_2)\\
\leq \ep^4 \Del (w_1^\ep - w_2^\ep).
\end{multline}
Let $\sig^\ep$ be the solution to 
\begin{equation*}
\begin{cases}
-\ep \sig^\ep_t +H_r(x,w^\ep_2,Dw^\ep_2) \sig^\ep - {\rm div}(D_pH(x,w^\ep_2,Dw^\ep_2) \sig^\ep)  = \ep^4 \Del \sig^\ep \quad &\text{in } \T^n \times (0,1),\\
\sig^\ep(x,1) = \del_{x_0},
\end{cases}
\end{equation*}
for $x_0 \in \T^n$ fixed.
Multiply \eqref{eq:diff} by $\sig^\ep$, integrate on $\T^n$ to imply
\[
\frac{d}{dt} \int_{\T^n} (w^\ep_1 - w^\ep_2)\sig^\ep\,dx \leq 0.
\]
In particular,
\[
(w_1^\ep - w_2^\ep)(x_0,1) \leq \int_0^1 \int_{\T^n} (w^\ep_1 - w^\ep_2)\sig^\ep\,dxdt
\]
By letting $\ep \to 0$ (and passing to a subsequence if needed) and using Lemma \ref{lem:ep}, we deduce that
\[
(u_1-u_2)(x_0) \leq \int_{\T^n} (u_1 - u_2) \,d\nu(x) \leq 0,
\]
for some $\nu \in \cM$.

\smallskip

Thus, $u_1(x_0) \leq u_2(x_0)$ for every $x_0 \in \T^n$.
The proof is complete.
\end{proof}

Set
\[
M := \overline{ \bigcup_{\nu \in \cM} {\rm supp} (\nu)}.
\]
Then we have the following corollary, which is an immediate consequence of Theorem \ref{thm:unique}.

\begin{cor}
Let $u_1, u_2$ be two solutions to \eqref{E-0}. 
Assume that $u_1 \leq u_2$ on $M$. 
Then, $u_1 \leq u_2$.
\end{cor}

\begin{rem}
Of course, this uniqueness result tells us that it is extremely important to have further understanding of the adjoint measures in $\cM$.
In case that $H(x,r,p) = \tilde H(x,p)$ for all $(x,r,p) \in \T^n \times \R \times \R^n$ with $\tilde H$ satisfies appropriate conditions, then these adjoint measures $\nu \in \cM$ turn out to be projected Mather measures (see \cite{FaB, MT-P}).

In the general setting, one objection one might have is that $\cM$ is defined in an abstract way, which depends on the set of solutions of \eqref{E-0} itself, 
and it is not clear how to analyze it.
This is a fair point, and $\cM$ should be studied much more in the near future.
In particular, one question of interests is whether $\cM$ can be defined without using $\cS$.

Nevertheless, in the following two interesting situations, we are able to provide full characterization of $\cM$.
These results are consistent with the classical literature of viscosity solutions, and also with Proposition \ref{prop:unique-0}.
 \end{rem}

\subsection{Adjoint measures in the strictly monotone case} \label{subsec:structure1}
We have the following result, which is consistent with the classical literature of viscosity solutions \cite{CEL, CL}.
\begin{prop}\label{prop:st-inc}
Assume that
\[
H_r(x,r,p) >0 \quad \text{ for all } (x,r,p)\in\T^n\times \R\times\R^n.  
\]
Then,
\[
\cM=\{0\}, \quad\text{and} \quad M=\emptyset. 
\]
\end{prop}
\begin{proof}
Let $u$ be a solution of \eqref{E-0}.
Let $w^\ep$ be the solution to \eqref{eq:C1} with initial data $u^{\ep^4}$.
Let $C_2=\|u\|_{L^\infty(\T^n)}+\|Du\|_{L^\infty(\T^n)}$.
Then, for $\ep<1$, we can find $C_3>0$ such that $\|w^\ep\|_{L^\infty(\T^n)}+\|Dw^\ep\|_{L^\infty(\T^n)} \leq C_3$.
By our hypothesis, we are able to find $\al>0$ such that
\[
H_r(x,w^\ep,Dw^\ep) \geq \al \quad \text{ for all } x\in \T^n.
\]
Use this in the adjoint equation to deduce
\[
-\ep\sig^\ep_t+\al\sig^\ep-\Div(D_pH(x,w^\ep,Dw^\ep)\sig^\ep)
\le \ep^4\Del \sig^\ep,   
\]
which implies 
\[
-\ep(e^{-\frac{\al t}{\ep}}\sig^\ep)_t-\Div(D_pH(x,w^\ep,Dw^\ep)e^{-\frac{\al t}{\ep}}\sig^\ep)
\le \ep^4\Del (e^{-\frac{\al t}{\ep}} \sig^\ep). 
\]
Integrate the above on $\T^n$ to obtain
\[
\dfrac{d}{dt}\left(\int_{\T^n}e^{-\frac{\al t}{\ep}}\sig^\ep\,dx\right)\ge0
\quad\text{ for all } \ t\in (0,1). 
\]
Thus, 
\[
\int_{\T^n}\sig^\ep(x,t) \,dx\le e^{-\frac{\al(1-t)}{\ep}} \quad\text{ for all } \ t\in (0,1),
\]
which gives
\[
\int_{\T^n}\,d\nu^\ep(x)
=\int_0^1\int_{\T^n}\sig^\ep(x,t) \,dxdt 
\leq \frac{\ep}{\al}(1-e^{-\frac{\al}{\ep}}).
\]
Sending $\ep\to0$ yields the conclusion. 
\end{proof}

Basically, Proposition \ref{prop:st-inc} says that in the strictly monotone setting, if $u_1, u_2$ are solutions to \eqref{E-0}, there is no need to compare $u_1$ and $u_2$ anywhere, and we have immediately $u_1=u_2$, which of course means that we have the unique viscosity solution to \eqref{E-0}. 

\subsection{Prototype example -- Revisit}\label{subsec:structure2}
Let us now revisit our prototype example in Section \ref{sec:prototype}.
Since we need smoothness of $H$, we consider
\[
H(x,r,p)=|p|^m - V(x) + f(r) \quad \text{ for } (x,r,p) \in \T^n \times \R \times \R^n.
\]
Here, $m \geq 2$ is a given number, and $V \in C^2(\T^n)$ is the potential energy with $\min_{\T^n} V=0$.
The function $f \in C^2(\R)$ is convex,  and
\[
\begin{cases}
f(r)=0 \quad &\text{ for } r \leq 0,\\
f(r)>0 \quad &\text{ for } r>0.
\end{cases}
\]
It is clear that $f'(r)>0$ for $r>0$, and $f'$ is nondecreasing.
In particular,
\begin{equation}\label{eq:f-p}
f'(s) \geq f'(r)>0 \quad \text{ for all } s\geq r>0.
\end{equation}
This observation will be used later on.
An example of $f \in C^2(\R)$ satisfying the above is $f(r) = (\max\{r,0\})^3$.

\begin{prop}\label{prop:prot}
Assume the setting in this subsection.
Then, for each $\nu \in \cM$,
\[
{\rm supp}(\nu) \subset M_V=\left\{x\in \T^n\,:\, V(x) = \min_{\T^n} V =0\right\}.
\]
\end{prop}

\begin{proof}
Let $u$ be a solution of \eqref{E-0}.
Let $w^\ep$ be the solution to \eqref{eq:C1} with initial data $u^{\ep^4}$.
The corresponding adjoint equation is
\[
-\ep \sig^\ep_t + f'(w^\ep) \sig^\ep -\Div(m |Dw^\ep|^{m-2} Dw^\ep \sig^\ep) = \ep^4 \Del \sig^\ep.
\]
Integrate this on $\T^n$ to get that
\[
\ep \frac{d}{dt} \int_{\T^n} \sig^\ep(x,t)\,dx  = \int_{\T^n} f'(w^\ep) \sig^\ep\,dx.
\]
Next, integrate the above in $t$ on $[0,1]$ to deduce further that
\[
\int_0^1 \int_{\T^n} f'(w^\ep) \sig^\ep \,dxdt = \ep \left(1 - \int_{\T^n} \sig^\ep (x,0)\,dx\right) \leq \ep.
\]
Letting $\ep  = \ep_j \to 0$ to yield, thanks to Lemma \ref{lem:ep},
\[
\int_{\T^n} f'(u)\, d\nu = 0.
\]
Thus, by using \eqref{eq:f-p}, we arrive at the fact that ${\rm supp} (\nu) \subset \{ u \leq 0\}$ for each $u \in \cS$.
We plan to pick an appropriate solution $u$ in $\cS$  to conclude.

Now, for each $c>0$, let $(u_c,c)$ be the unique solution to (E). 
Note that $u_c>0$.
Of course, for $c\in (0,1]$,
there exists $C>0$ independent of $c$ such that
\[
\|u_c\|_{L^\infty(\T^n)} + \|Du_c\|_{L^\infty(\T^n)}  \leq C.
\]
By using the Arzel\`a-Ascoli theorem, and passing to a subsequence if needed, $u_c \to u_0$ uniformly in $\T^n$ as $c \to 0$.
By stability of viscosity solutions, $u_0$ is a solution of \eqref{E-0}. 
It is clear that $u_0 \geq 0$, and 
\[
f(u_0) \leq V \quad \text{ in } \T^n,
\]
which gives that $\{V=0\} \subset \{u_0=0\}$.

Besides, for any $x_1 \in \T^n$ such that $u_0(x_1) = 0 = \min_{\T^n} u_0$, by the viscosity supersolution test,
\[
0=f(0)=f(u_0(x_1)) \geq V(x_1) \quad \Rightarrow \quad  V(x_1)=0.
\]
Thus, $\{V=0\} = \{u_0=0\}$, and hence, ${\rm supp} (\nu) \subset M_V=\{V=0\}$.
\end{proof}

This last proposition is consistent with the result of Proposition \ref{prop:unique-0}.
It is clear that we get $M \subset M_V$.
Nevertheless, we do not get that $M=M_V$ here, and it is not clear if this holds in general.
It would be very interesting if there is an example where $M \subsetneq M_V$.

\section*{Acknowledgments}

The work of WJ has been supported by the Recruitment Program of Global Experts of China and by the National Natural Science Foundation of China under Grant No.\,11701314.
The work of HM was partially supported by the JSPS grant KAKENHI \#16H03948.
The work of HT is partially supported by NSF grant DMS-1664424.

\bibliographystyle{amsplain}
\providecommand{\bysame}{\leavevmode\hbox to3em{\hrulefill}\thinspace}
\providecommand{\MR}{\relax\ifhmode\unskip\space\fi MR }
\providecommand{\MRhref}[2]{%
  \href{http://www.ams.org/mathscinet-getitem?mr=#1}{#2}
}
\providecommand{\href}[2]{#2}

\end{document}